\documentclass[10pt]{amsart}
\usepackage{amssymb,amsmath,geometry,latexsym,graphics,tabularx,shapepar,enumerate}
\usepackage[hidelinks]{hyperref}

\newtheorem{Theorem}{Theorem}
\newtheorem{Lemma}[Theorem]{Lemma}
\newtheorem{Proposition}[Theorem]{Proposition}
\newtheorem{Corollary}[Theorem]{Corollary}

\newtheorem{Remark}[Theorem]{Remark}
\newtheorem{Question}[Theorem]{Question}

\newtheorem{Conjecture}[Theorem]{Conjecture}

\newcommand{\AAA}{\mathcal{A}}

\newcommand{\NN}{\mathbb N}

\newcommand{\ZZ}{\mathbb Z}
\newcommand{\QQ}{\mathbb Q}

\newcommand{\RR}{\mathbb R}
\newcommand{\TT}{\mathbb T}
\newcommand{\CC}{\mathbb C}

\newcommand{\FF}{\mathbb F}
\newcommand{\HH}{\mathbb H}

\newcommand{\pitilde}{\widetilde{\pi}}

\newcommand{\undt}[1]{\underline{t}_{#1}}
 
\newcommand\CVD{{\hfill\hfil{\lower 2 pt\hbox{\vrule\vbox to 7pt 
{\hrule width 6pt\vfill\hrule}\vrule}}}\vskip 0.5cm}

\date{June 2016}

\title[$A$-harmonic sums]{On twisted $A$-harmonic sums and Carlitz finite zeta values}

\author{F. Pellarin \& R. Perkins}
\address{Federico Pellarin: Institut Camille Jordan, UMR 5208 Site de Saint-Etienne, 23 rue du Dr. P. Michelon, 42023 Saint-Etienne,
France}
\address{Rudolph Perkins: IWR, University of Heidelberg, Im Neuenheimer Feld 205, 69120 Heidelberg, Germany}
\email{federico.pellarin@univ-st-etienne.fr}
\email{rudolph.perkins@iwr.uni-heidelberg.de}
\thanks{The second author's research is supported by the Alexander von Humboldt Foundation}


\begin{document}

\begin{abstract}
In this paper, we study various twisted $A$-harmonic sums, named following the seminal log-algebraicity papers of G. Anderson. These objects are partial sums of new types of special zeta values introduced by the first author and linked to certain rank one Drinfeld modules over Tate algebras in positive characteristic by Angl\`es, Tavares Ribeiro and the first author. 
We prove, by using techniques introduced by the second author, that various infinite families of such sums may be interpolated by polynomials, and we deduce, among several other results, properties of analogues of finite zeta values but inside the framework of the Carlitz module.
In the theory of finite multi-zeta values in characteristic zero, finite zeta values are all zero. In the Carlitzian setting, there exist non-vanishing finite zeta values, and we study some of their properties in the present paper.
\end{abstract}

\maketitle

\section{Introduction}

Let $A=\FF_q[\theta]$ be the ring of polynomials in an indeterminate $\theta$  with coefficients in $\FF_q$ the finite field with 
$q$ elements and characteristic $p$, and let $K$ be the fraction field of $A$. We consider variables $t_1,\ldots,t_s$ over $K$ and we write $\undt{s}$
for the family of variables $(t_1,\ldots,t_s)$. We denote by $\boldsymbol{F}_s$ the field $\FF_q(\undt{s})$,
so that $\boldsymbol{F}_0=\FF_q$. 
For $d\geq 0$ an integer, we denote by $A^+(d)$ the set of monic polynomials of $A$
of degree $d$. We define the {\em twisted power sum of level $s$, degree $d$, and exponent $n$}
$$S_d(n;s)=\sum_{a\in A^+(d)}\frac{a(t_1)\cdots a(t_s)}{a^n}\in K[\undt{s}],$$
where, for a polynomial $a=\sum_ia_i\theta^i\in A$ with $a_i\in\FF_q$ and a variable $t$, $a(t)$ denotes the 
polynomial $\sum_ia_it^i$. 
If $s=0$ we recover the power sums already studied by several authors; see Thakur's \cite{THA}
and the references therein. For general $s$ these sums have been the object of 
study, for example, in the papers \cite{ANG&PEL,DEM}. We recall that,
in \cite[(3.7.4)]{AND&THA}, Anderson and Thakur proved, for all $n\geq 1$, that there exists a unique polynomial
$H_n\in A[Y]$ (with $Y$ an indeterminate) of degree in $Y$ which is at most $ \frac{nq}{q-1}$, such that, for all $d\geq 0$,
$$S_d(n;0)=\frac{H_n(\theta^{q^d})}{\Pi_nl_d^n},$$ where $\Pi_n$ is the $n$-th {\em Carlitz factorial} (see Goss' \cite[Chapter 9]{GOS}) and $l_d$ denotes $(-1)^d$ times the least common multiple of all polynomials of degree $d$; explicitly, $l_d$ is given by the product $(\theta-\theta^q)\cdots(\theta-\theta^{q^d})\in A$, for $d \geq 1$, and $l_0 := 1$. These investigations
have been generalized by F. Demeslay in his Ph. D. thesis \cite{DEM} to the sums $S_i(n;s)$ for any value of $s\geq 0$. 

In this text, we deal with a similar analysis, but for {\em twisted $A$-harmonic sums} as above, which are the sums:
$$F_d(n;s)=\sum_{i=0}^{d-1}S_{i}(n;s)\in K[\undt{s}],\quad n\in\ZZ,s\in\NN,d\in\NN\setminus\{0\}.$$
We are mainly interested in the case $n=1$ and $s\equiv1\pmod{q-1}$.
We denote by $b_i(Y)$ the product $(Y-\theta)\cdots(Y-\theta^{q^{i-1}})\in A[Y]$ (for an indeterminate $Y$) if $i>0$ and we set
$b_0(Y)=1$. We also write $m=\lfloor\frac{s-1}{q-1}\rfloor$ (the brackets denote the integer part so that $m$ is the biggest integer $\leq\frac{s-1}{q-1}$) (\footnote{We should have written $m_s$ instead of $m$,
to stress the dependence on $s$. However, eventually we will fix $s \equiv 1 \pmod{q-1}$, so we prefer to drop this subscript.}). We set
$$\Pi_{s,d}=\frac{b_{d-m}(t_1)\cdots b_{d-m}(t_s)}{l_{d-1}}\in K[\undt{s}],\quad d\geq \max\{1,m\}.$$

The main purpose of the present paper is to show the following result.
\begin{Theorem}\label{Theorem2}
For all integers $s\geq 1$, such that $s\equiv1\pmod{q-1}$, there exists a non-zero rational fraction $\mathbb{H}_{s}\in K(Y,\undt{s})$
such that, for all $d\geq m$, the following identity holds:
$$F_d(1;s)=\Pi_{s,d}\mathbb{H}_s|_{Y=\theta^{q^{d-m}}}.$$
If $s=1$, we have the explicit formula $$\mathbb{H}_1=\frac{1}{t_1-\theta}.$$
Further, if $s=1+m(q-1)$ for an integer $m>0$, then
the fraction $\mathbb{H}_s$ is a 
polynomial of $A[Y,\undt{s}]$ with the following properties:
\begin{enumerate}
\item For all $i$, $\deg_{t_i}(\mathbb{H}_s)=m-1$,
\item $\deg_{Y}(\mathbb{H}_s)=\frac{q^m - 1}{q-1} - m$.
\end{enumerate}
The polynomial $\mathbb{H}_s$ is uniquely determined by these properties.
\end{Theorem} 

We have found it somewhat subtle to compute the degree in $\theta$ of $\HH_s$, and we leave it as an open question. 

\subsection*{Limits and Carlitz zeta values in Tate algebras} We shall now explain the motivations
of this work.
We write $\CC_\infty$ for the completion $\widehat{K_\infty^{ac}}$
of an algebraic closure $K_\infty^{ac}$ of the completion $K_\infty$ of $K$ at the infinite place $\frac{1}{\theta}$.
One of the reasons for which we are interested in such properties of the sums $F_d(n;s)$ is that they
are the partial sums of the zeta-values
$$\zeta_A(n;s)=\lim_{i\rightarrow\infty}F_i(n;s)=\sum_{a\in A^+}\frac{a(t_1)\cdots a(t_s)}{a^n}, \quad n>0,\quad s\geq0,$$ with $A^+$ the set of monic polynomials of $A$,
converging in the {\em Tate algebra} $$\TT_s=\CC_\infty\widehat{\otimes}_{\FF_q}\FF_q[\undt{s}]$$ in the variables $\undt{s}$ and
with coefficients in $\CC_\infty$  (with the trivial valuation over $\FF_q[\undt{s}]$), introduced in \cite{PEL2} and 
studied, for example, in \cite{ANG&PEL,ANG&PEL2,APTR}. 
We choose once and for all a $q-1$-th root of $-\theta$ in $\CC_\infty$. Let 
$$\omega(t)=(-\theta)^{\frac{1}{q-1}}\prod_{i\geq 0}\left(1-\frac{t}{\theta^{q^i}}\right)^{-1},$$
 be the Anderson-Thakur function, in $\TT=\TT_1$ the Tate 
algebra in the variable $t=t_1$ with coefficients in $\CC_\infty$ (see \cite{AND&THA} for one of the first papers in which this function was studied and \cite{ANG&PEL2} for a more recent treatise of its basic properties). Let us also
consider the fundamental period $\widetilde{\pi}\in\CC_\infty$ of the $\FF_q[t]$-linear Carlitz exponential $\exp_C:\TT\rightarrow\TT$ (so that $\omega(t)=\exp_C(\frac{\widetilde{\pi}}{\theta-t})$, as in \cite{ANG&PEL2}).
We have the following result (see \cite{ANG&PEL}, see also \cite{APTR}), which gives one motivation for calling the special values $\zeta_A(n;s)$ zeta values.

\begin{Theorem}[B. Angl\`es and the first author]\label{anglespellarin}
For $s\equiv1\pmod{q-1}$ and $s>1$, there exists a polynomial $\lambda_{1,s}\in A[\undt{s}]$ such that
$$\zeta_A(1;s)=\frac{\widetilde{\pi}\lambda_{1,s}}{\omega(t_1)\cdots\omega(t_s)}.$$
\end{Theorem}
We also recall from \cite{PEL2} the formula
$$\zeta_A(1;1)=\frac{\widetilde{\pi}}{(\theta-t)\omega(t)}.$$
Theorem \ref{Theorem2} can be seen as a ``finite sum analogue" of Theorem \ref{anglespellarin}.
It is easy to see, by taking the limit $n\rightarrow\infty$ in the Tate algebra $\TT_s$, that Theorem \ref{Theorem2} implies Theorem \ref{anglespellarin}.
But Theorem \ref{Theorem2} contains more information; the process of convergence at the infinite place
which takes us to Theorem \ref{anglespellarin} from Theorem \ref{Theorem2} suppresses various information 
encoded in the formula of Theorem \ref{Theorem2}. We know from \cite{ANG&PEL} and \cite{APTR}
that $\lambda_{1,s}$ is a polynomial in $A[\undt{s}]$ when $s\geq q$ and $s\equiv1\pmod{q-1}$. For any such 
a choice of $s\equiv1\pmod{q-1}$, $\lambda_{1,s}$ is a generator of the Fitting ideal of a certain 
class module considered in \cite{APTR}, inspired by Taelman's theory in \cite{TAE2}. We will prove:

\begin{Theorem}\label{theomathbs}
For $s\geq q$ and $s\equiv1\pmod{q-1}$, $-\lambda_{1,s}$
is the leading coefficient of $\mathbb{H}_s$ as a polynomial in $Y$.
\end{Theorem}

\subsection*{Analytic formulas of lower coefficients}

We now discuss a similar analytic formula involving
the coefficients $\mathbb{D}_i\in A[\underline{t}_s]$ of the polynomial $$\mathbb{H}_s=\sum_{i=0}^\mu\mathbb{D}_iY^i\in A[\underline{t}_s][Y]$$ from Theorem \ref{Theorem2}. We recall that $\HH_s$ satisfies,
by Theorem \ref{Theorem2},
\begin{equation}\label{explicitidentity}
\frac{F_d(1;s)}{\Pi_{s,d}}=\frac{l_{d-1}F_d(1;s)}{b_{d-m}(t_1)\cdots b_{d-m}(t_s)}=\mathbb{H}_s(\theta^{q^{d-m}}),\end{equation}
for all $d \geq m$.
Since $\mu=0$ for $s=q$, we can restrict our attention to the case $s\geq 2q-1$ in this part.

We set, for all $d\geq m$ (\footnote{Here, $\TT_s(K_\infty)$ denotes the subring of the Tate algebra $\TT_s$ whose elements are formal series in $\underline{t}_s$ with coefficients in $K_\infty$.}),
$$\Gamma_d:=\frac{\prod_{i\geq d}\left(1-\frac{\theta}{\theta^{q^i}}\right)}{\prod_{i\geq d-m}\prod_{j=1}^s\left(1-\frac{t_j}{\theta^{q^i}}\right)}\in\TT_s(K_\infty).$$

In Lemma \ref{gammasr} below, we will define a sequence of polynomials $\Gamma_{s,r}\in A[\underline{t}_s][Y]$  monic of degree $\mu-r$ in $Y$
related to a series expansion of the product $\Gamma_d$. On the basis of this sequence of explicitly computable polynomials (we refer the reader to the statement of this lemma), we have the following result.
\begin{Theorem}\label{theoremhighercoefficients}
For all $r$ such that $0\leq r\leq \mu-1$, we have:
$$\mathbb{D}_r=-\lim_{d\rightarrow\infty}\left(\frac{\omega(t_1)\cdots\omega(t_s)}{\widetilde{\pi}}\Gamma_{s,r}(\theta^{q^{d-m}})F_d(1;s)
+\sum_{i=1}^{\mu-r}\mathbb{D}_{i+r}\theta^{iq^{d-m}}\right).$$
\end{Theorem}
It is easy to deduce, from Theorem \ref{theoremhighercoefficients},
an explicit identity for the coefficient $\mathbb{D}_{\mu-1}$ which of course makes sense only if $s\geq 2q-1$.
Indeed, it is easily seen that, with $r=\mu-1$, Theorem \ref{theoremhighercoefficients} and the simple identity $\Gamma_{s,\mu-1}=Y+\sum_{i=1}^st_i$ imply the next result. 
\begin{Corollary}
There exists a polynomial $\nu_{1,s}\in A[\underline{t}_s]$ such that
$$\lim_{d\rightarrow\infty}\theta^{q^{d-m}}\sum_{i\geq d}S_i(1;s)=\frac{\widetilde{\pi}\nu_{1,s}}{\omega(t_1)\cdots\omega(t_s)},$$ and $$\mathbb{D}_{\mu-1}=\nu_{1,s}-(t_1+\cdots+t_s)\lambda_{1,s}.$$
\end{Corollary}
We pose the following question.
\begin{Question} What is the arithmetic meaning of the lower coefficients $\mathbb{D}_r$, $r<\mu$?
Are they related in some way to Taelman's class modules, just as the coefficient $\mathbb{D}_\mu=-\lambda_{1,s}$
as shown in \cite{APTR}?
\end{Question}

\subsection*{Application to twisted power sums}
In his Ph. D. thesis \cite{DEM}, F. Demeslay recently proved that, for all $n\geq 1$ and $s'\geq 0$, there exists a unique 
rational fraction $Q_{n,s'}(\undt{s'},Y)\in K(\undt{s'},Y)$ such that, for some fixed integer
$r\geq0$ and for all $d$,
\begin{equation}\label{powersumsdemeslay} S_d(n;s')=l_d^{-n}b_d(t_1)\cdots b_d(t_{s'}) Q_{n,s'}(\undt{s'},\theta^{q^{d-r}}), \end{equation} hence providing a complement to \cite[(3.7.4)]{AND&THA}. Demeslay in fact proves 
several other properties related to the sums $S_d(n;s')$, and we refer the reader to his thesis ibid. for the details of his results.
Here we deduce, from Theorem \ref{Theorem2}, results similar to Demeslay's, though we do not make any explicit comparison with his results in this note.

We consider an integer $s\geq q$ such that $s\equiv1\pmod{q-1}$. We set $m=\frac{s-1}{q-1}$, and we choose an integer $s'$ such that $0\leq s'<s$.
We show that there exist {\em universal formulas} for the twisted power sums $S_{d-1}(1;s')$
 (this was also obtained by Demeslay) together
with a way to make them explicit by using the polynomials $\mathbb{H}_s$ (which is new compared to the previous literature).
For this, we set $\mathbb{H}_{s,s'}$ to be the coefficient of $t_{s'+1}^{m-1}\cdots t_{s}^{m-1}$
in  $\mathbb{H}_s$. It is a polynomial of $A[Y][\undt{s'}]$. It is easy to see that
the coefficient of $t_1^{m-1}\cdots t_s^{m-1}$ in $\mathbb{H}_s$ is one. Hence, the coefficient
of $t_1^{m-1}\cdots t_{s'}^{m-1}$ in $\mathbb{H}_{s,s'}$ is also equal to one. 

\begin{Theorem}\label{proppowersums}
Let $s = 1 + m(q-1)$, with $m \geq 1$, and let $0 \leq s' < s$. For each $d\geq m-1$, we have that
$$S_{d}(1;s')=l_d^{-1}\prod_{i=1}^{s'}b_{d+1-m}(t_i)\mathbb{H}_{s,s'}|_{Y=\theta^{q^{d+1-m}}}.$$
\end{Theorem}
This result also implies similar, but messier formulas for general twisted power sums $S_{d}(n;s')$
but we do not mention them here.

\subsection{Finite zeta values}\label{classicalcase} 
Theorem \ref{Theorem2} also opens a way through a study of {\em finite zeta and multi-zeta values}, which motivates the title of the paper. 

We consider, following Zagier, the ring
$$\AAA_\QQ=\frac{\displaystyle{\prod_p\frac{\ZZ}{p\ZZ}}}{\displaystyle{\bigoplus_p\frac{\ZZ}{p\ZZ}}},$$
the product and the direct sum running over the prime numbers $p$ (\footnote{It can be considered as
a kind of residue ring  of the rational ad\`eles $\mathbb{A}_\QQ$. Indeed,
there is a natural ring epimorphism
$\mathbb{A}_\QQ\rightarrow\AAA_\QQ$ sending an ad\`ele $(x_p)_p$
to the well defined residue $(x_p\pmod{p})_p$.}).  
Two elements
$(a_p)_p$ and $(b_p)_p\in\AAA_\QQ$ are equal if and only if $a_p=b_p$ for all
but finitely many $p$. The ring $\AAA_\QQ$ is not a domain. However, there is a natural injective ring homomorphism
$$\QQ\rightarrow\AAA_\QQ$$
defined by sending $r\in \QQ$ to the class modulo $\oplus_p\frac{\ZZ}{p\ZZ}$ of the sequence of its reductions mod $p$, well defined for almost all $p$ (that is, for all but finitely many $p$). Therefore, $\AAA_\QQ$ is a $\QQ$-algebra. This algebra is the main recipient for the 
theory of {\em finite multiple zeta values}, as in Kaneko's \cite{KAN}. 
For example, for all $k>0$, the finite zeta value of exponent $k$ trivially vanishes in $\AAA_\QQ$: 
$$\zeta_\AAA(k):=\left(\sum_{n=1}^{p-1}\frac{1}{n^k}\pmod{p}\right)_p=0\in\AAA_\QQ.$$
More generally, the finite multiple zeta values are defined, for $k_1,\ldots,k_r\in\ZZ$, by
$$\zeta_\AAA(k_1,\ldots,k_r)=\left(\sum_{0<n_1<\cdots<n_k<p}\frac{1}{n_1^{k_1}\cdots n_r^{k_r}}\pmod{p}\right)_p\in\AAA_\QQ.$$
As a matter of a fact, no one of these elements is currently known to be non-zero, and it is a basic problem in the theory to establish such non-vanishing. We discuss these issues further in the final section, \S \ref{finalsection}.

\subsection{Finite zeta values in the Carlitzian setting} We will apply Theorem \ref{Theorem2}
to show that certain variants of finite zeta values, introduced just below, are non-zero. 

As an analogue of the ring $\AAA_\QQ$, we consider the ring
$$\AAA_s :=\frac{\displaystyle{\prod_P\frac{\boldsymbol{F}_s[\theta]}{P\boldsymbol{F}_s[\theta]}}}{\displaystyle{\bigoplus_P\frac{\boldsymbol{F}_s[\theta]}{P\boldsymbol{F}_s[\theta]}}},$$
where the product and the direct sum run over the primes of $A$ (that is, irreducible monic polynomials of $A$). The indexation of the variables $t_1,\ldots,t_s$
induce embeddings $\AAA_0\hookrightarrow\AAA_1\hookrightarrow\cdots\hookrightarrow\AAA_s$ and in the following, we are viewing the rings $\AAA_s$ embedded one in the other as above. 
Let $K^{1/p^\infty}$ be the subfield of an algebraic closure $K^{ac}$ of $K$ whose elements $x$ are such that $x^{p^i}\in K$
for some $i$ (note that this is equal to the subfield of $K^{ac}$ whose elements are the $x$ such that $x^{q^i}\in K$ for some $i$). There is a natural embedding
$$K^{1/p^\infty}\otimes_{\FF_q}\boldsymbol{F}_s\xrightarrow{\iota}\AAA_s$$
(see \S \ref{iota} of the present paper) (\footnote{We prefer to keep our notations 
as simple as possible, and we avoid to denote the above embedding by $\iota_s$.}).
Let $P$ be a prime of degree $d$ in $\theta$; we extend the $P$-adic valuation $v_P$
of $K$ to $K\otimes_{\FF_q}\boldsymbol{F}_s$ by setting it to be the trivial valuation 
on $\boldsymbol{F}_s$.
Then, 
$v_P(S_i(n;s))\geq 0$ for all $0\leq i<d$ so that
$v_P(F_{d}(n;s))\geq 0$
for all $n\in\ZZ$ and $s\in\NN$. In particular, we have the {\em finite zeta value of level $s$ and exponent $n$},
$$Z_{\AAA}(n;s) :=\left(F_{\deg_\theta(P)}(n;s)\pmod{P}\right)_P\in\mathcal{A}_s.$$

The 
analysis of several examples of the zeta elements $Z_\AAA(n;s)$ leads us to formulate the 
following vanishing Conjecture, in which $n,s$ are integers such that $n>0$ and $s\geq 0$, and 
where, if $n=n_0+n_1q+\cdots+n_rq^r$ is the base-$q$ expansion of $n$, $\ell_q(n):=\sum_in_i$.

\begin{Conjecture}\label{nonvanishingconjecture}
The following properties hold.
\begin{enumerate}
\item If $n\not\equiv s\pmod{q-1}$, then $Z_\AAA(n;s)\neq0$.
\item If $n\equiv s\pmod{q-1}$, then $Z_\AAA(n;s)=0$ if and only if $\ell_q(n)>s$.
\end{enumerate}
\end{Conjecture}

It is easy to deduce, in the case $q>2$, from \cite[Lemma 3.5]{ANG&DOU}, that $Z_\AAA(1;0)\in\AAA_0\setminus\{0\}$
agreeing with (1) of the conjecture.
In Lemma \ref{casesnot0} below, we show that the case (1) of the conjecture 
holds true if $s>0$. 
In fact, the case $n=1$ of the conjecture is completely settled including in the case (2), more difficult, as we are going to 
see now.
We consider the following elements, 
$$\widehat{\pi} :=\left(-\frac{1}{P'}\right)_P\in\AAA_0^\times,\quad \widehat{\omega}(t) :=\left(\frac{1}{P(t)}\right)_P\in\AAA_1^\times, \text{ and } \quad \widehat{\Pi}_{n,s} :=\frac{\widehat{\pi}^n}{\widehat{\omega}(t_1)\cdots\widehat{\omega}(t_s)}\in\AAA_s^\times,$$
where we write $t=t_1$ if $s=1$ and $n,s$ are integers such that $n\equiv s\pmod{q-1}$
(if $R$ is a unitary ring, $R^\times$ denotes the group of invertible elements of $R$).
Here, the dash $'$ denotes the derivative with respect to $\theta$ in $K$. 
It is easy to show that they are indeed units of the respective rings; further basic properties of $\widehat{\pi}$ and $\widehat{\omega}$ are given in \S \ref{periodproperties}; for example,
we will prove in Theorem \ref{irrational} that $\widehat{\pi}$ is ``irrational", i.e. it does not belong to 
$\iota(K)$. 

Our main application of Theorem \ref{Theorem2} to finite zeta values is the following result.

\begin{Theorem}\label{Theorem1}
Assume that $s\equiv 1\pmod{q-1}$. Then, there exists a non-zero explicit element $\mu_{1,s}\in
K^{1/p^\infty}\otimes_{\FF_q}\boldsymbol{F}_s$  such that
$$Z_\AAA(1;s)=\widehat{\Pi}_{1,s}\iota(\mu_{1,s}).$$
In particular, $Z_\AAA(1;s) \in \AAA_s^\times$.
\end{Theorem}

We recall that the second author has shown \cite{PER2} that, for $n\geq 0$ and $s\geq0$, 
the sequence of the partial sums of the series $\sum_{d\geq 0}S_{d}(-n;s)$ is ultimately equal to a polynomial
$\operatorname{BG}(n;s)\in A[\undt{s}]$ called the {\em Bernoulli-Goss polynomial of order $n$ and level $s$}. The next result settles the ``if'' part of (2) of Conjecture \ref{nonvanishingconjecture}. We recall that $\ell_q(n)$ denotes the sum of the digits of the base-$q$ expansion of the positive integer $n$.

\begin{Theorem}\label{proposition1}
Let $n$ and $s$ be integers such that $n\geq1$ and $s\geq 0$. Suppose that $\ell_q(n) > s$. Then $Z_\AAA(n;s) = 
(\operatorname{BG}(q^{\deg_\theta(P)}-1-n;s)\pmod{P})_P$. In particular, if 
$s\equiv n\pmod{q-1}$, we have that $Z_\AAA(n;s)=0$.
\end{Theorem}
We deduce that $Z_\AAA(n;0)=0$ for all $n\geq 1$, with $n\equiv0\pmod{q-1}$.

\subsubsection*{Acknowledgements} The first author is thankful to Masanobu Kaneko for several fruitful discussions about his work on finite multiple zeta values, during his stay in Saint-Etienne and Lyon in September 2015, which helped and motivated this work. Both authors are indebted with David Goss for helpful suggestions and continuous encouragement. We gratefully acknowledge Bruno Angl\`es for helpful comments and for having drawn to our attention the reference \cite{ANG&DOU}. Finally, we thank A. Maurischat for his careful reading of the proof of Theorem 1 and valuable comments which improved and simplified the presentation.

\section{Proofs}

\subsubsection*{Basic notation}

\begin{itemize}
\item $\lfloor x\rfloor$: largest rational integer $n$ such that $n\leq x$, with $x$ real number.
\item $\log_q(x)$: logarithm in base $q$ of the real number $x>0$.
\item $\ell_q(n)$: sum of the digits of the expansion in base $q$ of the non-negative integer $n$.
\item $m$: the integer $\lfloor\frac{s-1}{q-1}\rfloor$ for $s\geq 1$.
\item $[n]$: The polynomial $\theta^{q^n}-\theta\in A$ with $n>0$ integer. We also set $[0]:=1$. 
\item $l_n$: the sequence of $A$ defined by $l_0=1$ and $l_n=-[n]l_{n-1}$ with $n>0$.
\item $D_n$: the sequence of $A$ defined by $D_0=1$ and $D_n=[n]D_{n-1}^q$.
\item $b_n$: the polynomial $(Y-\theta)\cdots(Y-\theta^{q^{n-1}})\in A[Y]$ if $n>0$ and $b_0=1$,
for $n>0$.
\item $A(i)$: for $i\geq 1$ an integer, the $\FF_q$-vector space of polynomials
of $A$ of degree $<i$ in $\theta$. We also set $A(0)=\{0\}$.
\item $\boldsymbol{F}_s$: the field $\FF_q(\undt{s})=\FF_q(t_1,\ldots,t_s)$.
\end{itemize}

\subsection{Proof of Theorem \ref{Theorem2}}\label{beggining}

We first consider the separate case of $s=1$ and we set $t=t_1$. 
It is easy to see that in $K[t]$, for all 
$d\geq 1$, 
\begin{equation}\label{ide1}
F_d(1;1)=\sum_{i=0}^{d-1}\sum_{a\in A^+(i)}a^{-1}a(t)=\frac{b_d(t)}{(t-\theta)l_{d-1}},
\end{equation}
see for instance \cite{ANG&PEL2}.
This proves Theorem \ref{Theorem2} with $s=1$, $m=0$ and $\mathbb{H}_1=\frac{1}{t-\theta}\in K(t)$.

\medskip

In all the rest of this subsection, we suppose that $s \geq q$ is an integer such that $s\equiv1\pmod{q-1}$, and we let $m = \frac{s-1}{q-1}$.
We proceed in several steps: in \S \ref{step1} we analyze the dependence on $d$ of the relations and 
the occurrence of the variable $Y$, in \S \ref{step2} we study the interpolation properties of certain series
and conclude the proof of the first part of Theorem 
\ref{Theorem2}, namely, that $\mathbb{H}_s$ exists and belongs to $K(Y)[\undt{s}]$.
Finally, in \S \ref{step4}, we conclude the proof of Theorem \ref{Theorem2} by using a Zariski density argument.

\subsubsection{Existence of universal relations}\label{step1}

We shall write, for a new indeterminate $z$
$$E_i=E_i(z) :=D_i^{-1}\prod_{a\in A(i)}(z-a)\in K[z],$$ where $A(i)$ denotes the $\FF_q$-vector space
of polynomials of $A$ whose degree is strictly less than $i$ in $\theta$.
In Goss' \cite[Theorem 3.1.5]{GOS}, the following formula, due to Carlitz, is proved:
$$E_i(z)=\sum_{j=0}^i\frac{z^{q^j}}{D_jl_{i-j}^{q^j}},\quad i\geq0.$$
In particular, we deduce the following result that we will use later.
\begin{Proposition}\label{propositionE}
The following properties hold, for all $i\geq 0$.
\begin{enumerate}
\item The polynomial $E_i$ is $\FF_q$-linear of degree $q^i$ in $z$, and $E_i(\theta^i) = 1$. 
\item For all $a\in A(i)$ and all $i\geq 0$ we have 
$$\left.\frac{E_i(z)}{z-a}\right|_{z = a} = \frac{d}{dz}E_i = \frac{1}{l_i}.$$
\item For all $i\geq 0$, we have that $E_i^q=E_i+[i+1]E_{i+1}$. \hfill \qed
\end{enumerate}
\end{Proposition}

Next, if $n$ is a non-negative integer and if $n=n_0+n_1q+\cdots+n_rq^r$ is its base-$q$ expansion
(with the digits $n_0,\ldots,n_r\in\{0,\ldots,q-1\}$), we shall consider the polynomial
$$G_n=E_0^{n_0}\cdots E_r^{n_r}\in K[z],$$ so that $G_0=1$. By Proposition \ref{propositionE}, (1), $\deg_z(G_n)=n$
so that every polynomial $Q\in K[z]$ can be written, in a unique way, as a finite sum
$$Q=\sum_{n\geq 0}c_nG_n,\quad c_n\in K.$$

\begin{Lemma}[Universal Relations Lemma] \label{univcoeffslem}
Let $\underline{j} := (j_1,\dots,j_r)$ be an $r$-tuple of non-negative integers. There exist polynomials $\{c_{\underline{j},i}\}_{i \geq 0} \subset A[Y]$, all but finitely many of which are non-zero, such that for each non-negative integer $n$ we have
\[ E_{n+j_1}E_{n+j_2} \cdots E_{n+j_r} = \sum_{i \geq 0} \left. c_{\underline{j},i} \right|_{Y = \theta^{q^n}} G_{iq^n}. \]
\end{Lemma}

\begin{proof}
We induct on the number of entries $r$ of the tuple $\underline{j} = (j_1,j_2,\dots,j_r)$. After reindexing, we may assume that $0 \leq j_1 \leq j_2 \leq \cdots \leq j_r$. 

If no $q$ consecutive entries are equal, i.e. $j_{k+1} = j_{k+2} = \cdots = j_{k+q}$ does not hold for any $k$, we have
\[ E_{n+j_1}E_{n+j_2} \cdots E_{n+j_r} = G_{q^n(q^{j_1}+q^{j_2}+\dots+q^{j_r})},\]
and we obtain a single non-zero polynomial $c_{\underline{j}, \sum q^{j_i}}(Y) := 1$. 

So, assume that we have $j_{k+1} = j_{k+2} = \cdots = j_{k+q}$, for some $0 \leq k \leq r-q$. Grouping these $E_i$'s together and applying (3) of Proposition \ref{propositionE} above, we obtain
\begin{align} \label{recurpfeq1}
\nonumber E_{n+j_1}E_{n+j_2} \cdots E_{n+j_r}  = & E_{n+j_1} \cdots E_{n+j_k} (E_{n+j_{k+1}})^q E_{n+j_{k+q+1}} \cdots E_{n+j_r} \\
= & E_{n+j_1} \cdots E_{n+j_k} E_{n+j_{k+1}} E_{n+j_{k+q+1}} \cdots E_{n+j_r} \\ 
\nonumber & + [n+j_{k+1}+1] E_{n+j_1} \cdots E_{n+j_k} E_{n+j_{k+1}+1} E_{n+j_{k+q+1}} \cdots E_{n+j_r}.
\end{align}

Now, both
\[e_1 := E_{n+j_1} \cdots E_{n+j_k} E_{n+j_{k+1}} E_{n+j_{k+q+1}} \cdots E_{n+j_r} \text{ and} \] 
\[e_2 := E_{n+j_1} \cdots E_{n+j_k} E_{n+j_{k+1}+1} E_{n+j_{k+q+1}} \cdots E_{n+j_r},\] which occur in the previous displayed line, both come from tuples with $r-(q-1)$ many entries, namely 
\[ \underline{j}_1 := (j_1,\dots,j_{k+1},j_{k+q+1},\dots,j_r) \quad \text{ and } \quad \underline{j}_2 := (j_1,\dots,j_{k},j_{k+1}+1,j_{k+q+1},\dots,j_r)  \] 
and hence by induction we deduce the existence of polynomials $\{ c_{\underline{j}_1,i}(Y)\}_{i \geq 0}$ and $\{c_{\underline{j}_2,i}(Y)\}_{i \geq 0}$ in $A[Y]$, all but finitely many of which are non-zero, such that
\[e_1 = \sum_{i \geq 0} c_{\underline{j}_1,i}(\theta^{q^n}) G_{i q^n} \text{ and } e_2 = \sum_{i \geq 0} c_{\underline{j}_2,i}(\theta^{q^n}) G_{i q^n},\]
for all $n \geq 0$. 

Returning with this to \eqref{recurpfeq1}, we obtain
\[E_{n+j_1}E_{n+j_2} \cdots E_{n+j_r} = \sum_{i \geq 0} \left( c_{\underline{j}_1,i}(Y) + (Y^{q^{j_{k+1}+1}} - \theta) c_{\underline{j}_2,i}(Y)\right)|_{Y = \theta^{q^n}} G_{i q^n}.\]
So we let $c_{\underline{j},i}(Y) := c_{\underline{j}_1,i}(Y) + (Y^{q^{j_{k+1}+1}} - \theta) c_{\underline{j}_2,i}(Y)$.
\end{proof}

\subsubsection{Interpolation properties}\label{step2}

Recall that $A(d)$ denotes the $\FF_q$-vector space of the polynomials of $A$ of degree strictly less than $d$ in $\theta$.
We consider, for $d\geq 1$, the element
$$\psi_{s,d} :=\sum_{a\in A(d)}\frac{a(t_1)\cdots a(t_s)}{z-a}\in\boldsymbol{F}_s\otimes_{\FF_q}K(z).$$
By Proposition \ref{propositionE} properties (1) and (2), 
$$N_{s,d}:=l_dE_d\psi_{s,d}$$
is the unique polynomial of $K[z,\undt{s}]$ of degree $<q^d$ in $z$ 
such that the associated map $\CC_\infty\rightarrow\CC_\infty[\undt{s}]$ which sends $z$ to the polynomial $N_{s,d}(z)$ interpolates the
map 
$$A(d) \ni a\mapsto a(t_1)\cdots a(t_s)\in \FF_q[\undt{s}].$$
In \cite{PER3}, the second author found several explicit formulas for the sums $\psi_s=\psi_{s,\infty}=\lim_{d\rightarrow\infty}\psi_{s,d}$ for small values of $s$ and in \cite{PEL&PER} the two authors 
of the present paper improved qualitatively the previous results without any restriction on $s$. In particular, the following formula holds, with $t=t_1$, valid for $d\geq 1$:
\begin{equation}\label{simplestcase}
N_{1,d}=\sum_{j=0}^{d-1}E_j(z)b_j(t)\in K[z,t].
\end{equation}
We also set $M_{s,d}=\prod_{i=1}^s(N_{1,d})_{t=t_i}$ and notice that, 
since $\deg_z(N_{1,d})=q^{d-1}$ by (1) of Proposition \ref{propositionE}, the degree in $z$ of $M_{s,d}$ is equal to $sq^{d-1}$. We can thus write, for $d\geq 1$:
\begin{eqnarray*}
M_{s,d} &=& \sum_{\underline{i}\leq d-1}E_{i_1}(z)\cdots E_{i_s}(z)b_{i_1}(t_1)\cdots b_{i_s}(t_s)\\
&=&\sum_{\underline{i} \leq d-1}\sum_{0\leq j\leq sq^{d-1}}\epsilon_{\underline{i},j}G_j(z)b_{i_1}(t_1)\cdots b_{i_s}(t_s)
\end{eqnarray*}
where, we expand 
\begin{equation} \label{epsilondefeq}
E_{i_1}\cdots E_{i_s}= \sum_{j}\epsilon_{\underline{i},j}G_j,
\end{equation} 
with $\epsilon_{\underline{i},j}\in K$, and where, $\underline{i} \leq n$ stands for the inequalities $i_j\leq n$ for $j=1,\ldots,s$ (and similarly for $\underline{i} \geq n$).

Since both $M_{s,d}$ and $N_{s,d}$ have the same interpolation property
(they interpolate the map $a\mapsto a(t_1)\cdots a(t_s)$ over $A(d)$), the polynomial 
$M_{s,d} - N_{s,d}$ vanishes for all $z=a\in A(d)$. This means that $E_d$ divides $M_{s,d} - N_{s,d}$.
But $\deg_z(N_{s,d})<q^d = \deg_z E_d$ from which we deduce that, for $d\geq 1$,
$$M_{s,d} - N_{s,d}=
\sum_{\underline{i} \leq d-1}\sum_{q^d\leq j\leq sq^{d-1}}\epsilon_{\underline{i},j}G_j(z)b_{i_1}(t_1)\cdots b_{i_s}(t_s)
,$$
and, in particular,
\[ N_{s,d} = \sum_{\underline{i} \leq d-1}\sum_{0\leq j\leq q^d - 1}\epsilon_{\underline{i},j}G_j(z)b_{i_1}(t_1)\cdots b_{i_s}(t_s). \]
Now we notice that, for $s\geq 2$, we have that $(M_{s,d}/E_d)_{z=0}=0$. On one hand, 
we have (recall that $s\equiv1\pmod{q-1}$):
$$\left(\frac{M_{s,d} - N_{s,d}}{l_dE_d}\right)_{z=0}= -\psi_{s,d}(0)= -\sum_{a\in A(d)}\frac{a(t_1)\cdots a(t_s)}{-a} = -F_d(1;s).$$
On the other hand we note that, for all $j>0$, if $\ell_q(j)\neq1$, then $(G_j/l_dE_d)_{z=0}=0$, and if $\ell_q(j)=1$,
then, for some $k\geq 0$, $j=q^k$ and $G_j=E_k$. Further, we have
$$\left(\frac{E_k}{l_dE_d}\right)_{z=0}=\frac{D_d}{D_k}\frac{\prod_{0\neq a\in A(k)}a}{\prod_{0\neq a\in A(d)}a}=
\frac{1}{l_k},$$
by \cite[\S 3.2]{GOS} or by (2) of Proposition \ref{propositionE}. 
Thus, we obtain
\begin{eqnarray*}
-F_d(1;s)
&=&\sum_{d\leq k\leq d-1+\lfloor\log_q(s)\rfloor}l_k^{-1}\sum_{\underline{i} \leq d-1} \epsilon_{\underline{i},q^k}
b_{i_1}(t_1)\cdots b_{i_s}(t_s)\\
&=&\sum_{\underline{i} \leq d-1}b_{i_1}(t_1)\cdots b_{i_s}(t_s)\sum_{d\leq k\leq d-1+\lfloor\log_q(s)\rfloor}l_k^{-1}
\epsilon_{\underline{i},q^k}.
\end{eqnarray*}
By \cite[Proposition 10]{ANG&PEL} we see that, for all $i=1,\ldots,s$ and all $d\geq m$, if $r=0,\ldots,d-m-1$, then $F_d(1;s)|_{t_i=\theta^{q^r}}=0$
 (if $d=m$ the property is void). Since the family of polynomials $(b_{i_1}(t_1)\cdots b_{i_s}(t_s))_{\underline{i}\geq 0}$
is a basis of the $K$-vector space $K[\undt{s}]$, this means that, for all $d\geq m$,
$$-F_d(1;s)=\sum_{d-m\leq\underline{i} \leq d-1}b_{i_1}(t_1)\cdots b_{i_s}(t_s)\sum_{d\leq k\leq d+\lfloor\log_q(s)\rfloor - 1}l_k^{-1}\epsilon_{\underline{i},q^k}.$$
We can rewrite this identity as follows, for all $d\geq m$:
\begin{equation}\label{fdexpr}
-F_d(1;s)=\sum_{h=0}^{\lfloor\log_q(s)\rfloor-1}\frac{1}{l_{d+h}} \sum_{0\leq j_1,\ldots,j_s\leq m}
\epsilon_{(d-m+j_1,\ldots,d-m+j_s),q^{d+h}}b_{d-m+j_1}(t_1)\cdots b_{d-m+j_s}(t_s),
\end{equation}
and this puts us in the situation of Lemma \ref{univcoeffslem}, with $n = d - m$, $r = s$ and $\underline{j} = (j_1,\dots,j_s)$. Thus, we see from \eqref{epsilondefeq} and the aforementioned lemma that 
\[\epsilon_{(d-m+j_1,\ldots,d-m+j_s),q^{d+h}} = c_{\underline{j},q^{h+m}}|_{Y = \theta^{q^{d-m}}}.\] 
Finally, for each $h$ and $\underline{j}$, as above, we have
\[\frac{l_{d-1}}{\prod_{i = 1}^s b_{d-m}(t_i)}\frac{\prod_{i = 1}^sb_{d-m+j_i}(t_i)}{l_{d+h}} = \left. \frac{\prod_{i = 1}^s (t_i - Y)(t_i - Y^q)\cdots (t_i - Y^{q^{j_i-1}}) }{(\theta- Y^{q^m})(\theta- Y^{q^{m+1}})\cdots (\theta- Y^{q^{m+h}})}\right|_{Y = \theta^{q^{d-m}}}. \]
Thus, letting $w_{\underline{j},s} := \frac{\prod_{i = 1}^s (t_i - Y)(t_i - Y^q)\cdots (t_i - Y^{q^{j_i-1}}) }{(\theta- Y^{q^m})(\theta- Y^{q^{m+1}})\cdots (\theta- Y^{q^{m+h}})}$, we obtain
\begin{equation} \label{qualitativeformula} \frac{l_{d-1}}{\prod_{i = 1}^s b_{d-m}(t_i)} F_d(1;s) = -\sum_{h=0}^{\lfloor\log_q(s)\rfloor-1} \sum_{0\leq j_1,\ldots,j_s\leq m} (c_{\underline{j},q^{h+m}} w_{\underline{j},s})|_{Y = \theta^{q^{d-m}}}, 
\end{equation}
completing the proof of the first part of Theorem \ref{Theorem2}, namely, that
$\mathbb{H}_s$ exists, and is a rational fraction of $K(Y)[\undt{s}]$.

\begin{Remark}{\em We provide, additionally, formulas which show how the 
functions $\psi_{s,d}$ can be viewed as generating series of the sums $F_d(n;s)$.
Let $d>0 $ be fixed. If $z\in\CC_\infty$ is such that
$|z|>q^{d-1}$, then, for all $a\in A$ with $\deg_\theta(a)<d$, we have
$\frac{1}{z-a}=\frac{1}{z}\frac{1}{1-\frac{a}{z}}=\frac{1}{z}\sum_{i\geq 0}\left(\frac{a}{z}\right)^i$.
Hence,
$$\psi_{s,d}(z)=\sum_{a\in A(d)}\frac{a(t_1)\cdots a(t_s)}{z-a}=-\sum_{\tiny{\begin{array}{c}i\geq 0\\ s+i\equiv0 \\ \pmod{q-1}\end{array}}}z^{-1-i}F_d(-i;s).$$
Similarly, if $s>0$, $|z|<1$ and $a\neq0$, we have that 
$\frac{1}{z-a}=-\frac{1}{a}\frac{1}{1-\frac{z}{a}}=-\frac{1}{a}\sum_{j\geq 0}\left(\frac{z}{a}\right)^j$
and $$\psi_{s,d}(z)=\sum_{\tiny{\begin{array}{c}j\geq 0\\ j+1-s\equiv0 \\ \pmod{q-1}\end{array}}}z^jF_d(1+j;s).$$
 }\end{Remark}

\subsubsection{Conclusion of the proof of Theorem \ref{Theorem2}}\label{step4}

\begin{Proposition}\label{simple} Assume that $s\geq q$ and $s\equiv1\pmod{q-1}$. Then, for all 
$d\geq m$, we have that
$$\mathbb{H}_s|_{Y=\theta^{q^{d-m}}}=H_{s,d},$$ where $H_{s,d}$ is a non-zero polynomial of
$A[\underline{t}_s]$ of degree $m-1=\frac{s-q}{q-1}$
in $t_i$ for all $i=1,\ldots,s$.
\end{Proposition}

\begin{proof}
 We note that $l_{d-1}F_d(1;s)\in A[\undt{s}]$ and that, for $d\geq m$,
the polynomial $l_{d-1}F_d(1;s)$ is divisible by $b_{d-m}(t_1)\cdots b_{d-m}(t_s)$ in virtue of 
\cite[Proposition 10]{ANG&PEL}. 

One easily calculates, for each $1 \leq i \leq s$, that the degree in $t_i$ of $F_d(1;s)$ equals $d-1$, and the degree of $b_{d-m}(t_i)$ in $t_i$ equals $d-m$. Since $F_d(1;s)/\Pi_{s,d} \in A[\underline{t}_s]$, we must have that the degree in $t_i$ of this ratio is $d-1 - (d-m) = m-1.$
\end{proof}

\begin{Remark} \label{degHsd}
{\em The degree of $H_{s,d}$ in $\theta$ can be easily deduced, for all $d \geq m$, from the arguments in the proof of \cite[Proposition 11]{ANG&PEL}.
We obtain for this degree 
$$\delta_{s,d}:=\frac{q^d-q}{q-1}-s\frac{q^{d-m}-1}{q-1}= m-1+\mu q^{d-m}$$ in $\theta$, where $\mu = \frac{q^m - 1}{q-1} - m$.
To verify the displayed formula for $\delta_{s,d}$, write $s=m(q-1)+1$ and observe that $m\geq 1$ because $s\geq q$. Since we will not use the degree in $\theta$ of $H_{s,d}$ for small $d$ to follow, we do not give the details of this argument here. }
\end{Remark}

\subsubsection*{Integrality of $\mathbb{H}_s$}

We need an elementary fact.

\begin{Lemma}\label{lemmau}
Let $\mathbb{U}=\mathbb{U}(Y)$ be a polynomial of $A[Y][\undt{s}]$ and $M\geq 0$ and integer such that, for all 
$d$ big enough, 
$$\frac{\mathbb{U}(\theta^{q^d})}{\theta^{q^{d+M}}-\theta}\in A[\undt{s}].$$
Then, $\mathbb{U}=(Y^{q^M}-\theta)\mathbb{V}$ with $\mathbb{V}\in A[Y][\undt{s}]$.
\end{Lemma}

\begin{proof} If $\mathbb{U}$ has degree $<q^M$ in $Y$, then, for all $d$ big enough,
$\deg_\theta(\mathbb{U}(\theta^{q^d}))\leq C+q^d(q^M-1)$ with $C$ a constant depending on $\mathbb{U}$.
Therefore, as $d$ tends to infinity,
$$\deg_\theta\left(\frac{\mathbb{U}(\theta^{q^d})}{\theta^{q^{d+M}}-\theta}\right)\leq C+q^d(q^M-1)-q^{d+M}\rightarrow-\infty,$$ which implies that $\mathbb{U}(\theta^{q^d})=0$ for all $d$ big enough, and 
$\mathbb{U}=0$ identically because the set $\{\theta^{q^d};d\geq d_0\}\subset\CC_\infty$ is Zariski-dense for all 
$d_0$. This proves the Lemma in this case.

Now, if $\mathbb{U}$ has degree in $Y$ which is $\geq q^M$, we can write, by 
euclidean division (the polynomial $Y^{q^M}-\theta$ is monic in $Y$), $\mathbb{U}=(Y^{q^M}-\theta)\mathbb{V}+\mathbb{W}$ with 
$\mathbb{V},\mathbb{W}\in A[Y][\undt{s}]$ and $\deg_Y(\mathbb{W})<q^M$.
Hence, by the first part of the proof, $\mathbb{W}=0$.
\end{proof}

First we show that $\mathbb{H}_s\in A[Y][\undt{s}]$.
Indeed, by (\ref{qualitativeformula}), we can write
$$\mathbb{H}_s(Y)=\frac{\mathbb{U}(Y)}{(\theta-Y^{q^m})\cdots(\theta-Y^{q^{m+\kappa_0}})},$$
where $\kappa_0=\lfloor\log_q(s)\rfloor-1$ and $\mathbb{U}$ is a polynomial in $A[Y][\undt{s}]$. By Proposition \ref{simple},  $\HH_s(\theta^{q^{j}}) \in A[\underline{t}_s]$, for all $j \geq 0$, and we deduce that $\frac{\mathbb{U}(\theta^{q^j})}{(\theta-\theta^{q^{j+k}})} \in A[\underline{t}_s]$, for each $m \leq k \leq m+\kappa_0$ and all $j \geq 0$.
By Lemma \ref{lemmau}, we conclude that $\frac{\mathbb{U}(Y)}{(\theta-Y^{q^{k}})} \in A[\underline{t}_s,Y]$, for each $m \leq k \leq m + \kappa_0$.
Since the polynomials $\theta-Y^{q^{m}}, \theta-Y^{q^{m+1}}, \dots, \theta-Y^{q^{m+\kappa_0}}$ are relatively prime, it follows that $\mathbb{H}_s\in A[Y][\undt{s}]$, as claimed. 

We end the proof of Theorem \ref{Theorem2} by verifying the quantitative data using Proposition \ref{simple}. We set $\nu=m-1$ (note that $\mu=0$ in case $s=q$), so that 
$\deg_{t_i}(H_{s,d})=\nu$, for all $d\geq m$. Writing $\mathbb{H}_s=\sum_{\underline{i}}c_{\underline{i}}(Y)\underline{t}^{\underline{i}}$
with $c_{\underline{i}}(Y)\in A[Y]$ and $H_{s,d}=\sum_{\underline{i}}c'_{\underline{i},d}\underline{t}^{\underline{i}}$
with $c'_{\underline{i},d}\in A$
(\footnote{We are adopting multi-index notations, so that, if $\underline{i}=(i_1,\ldots,i_s)\in\NN^s$, then 
$\underline{t}^{\underline{i}}=t_1^{i_1}\cdots t_s^{i_s}$.}),
we have $$c_{\underline{i}}(\theta^{q^{d-m}})=c'_{\underline{i},d},$$
for all $d\geq m$ and for all $\underline{i}$. This means that $\deg_{t_i}(\mathbb{H}_s)=\deg_{t_i}(H_{s,d})=\nu$
for all $d\geq m$ and $i=1,\ldots,s$. This confirms the data on the degree in $t_i$ for all $i$.
Now that we know $\HH_s \in A[\underline{t}_s,Y]$, the claim on the degree of $\HH_s$ in $Y$ follows from the proof of Theorem \ref{theomathbs} below; specifically, see Lemma \ref{lemmalimit}. \hfill \CVD

\subsection{Proof of Theorem \ref{theomathbs}} We suppose that $s\geq q$ is an integer such that $s\equiv1\pmod{q-1}$.
We know from \cite{ANG&PEL,APTR} that, in Theorem \ref{anglespellarin},
$$\mathbb{B}_s=(-1)^{\frac{s-1}{q-1}}\lambda_{1,s}\in A[\underline{t}_s],$$ a polynomial which is monic of degree $\frac{s-q}{q-1}$ in $\theta$. We choose a root $(-\theta)^{\frac{1}{q-1}}$ of $-\theta$ and we set:
\begin{eqnarray*}
\widetilde{\pi}_d&:=&\theta(-\theta)^{\frac{1}{q-1}}\prod_{i=1}^{d-1}\left(1-\frac{\theta}{\theta^{q^i}}\right)^{-1}\in (-\theta)^{\frac{1}{q-1}} K^\times,\\
\omega_d(t)&:=&(-\theta)^{\frac{1}{q-1}}\prod_{i=0}^{d-1}\left(1-\frac{t}{\theta^{q^i}}\right)^{-1}\in(-\theta)^{\frac{1}{q-1}}  K(t)^\times.\end{eqnarray*}
Then, in $\TT_s$, $$\lim_{d\rightarrow\infty}\frac{\widetilde{\pi}_d}{\omega_{d-m}(t_1)\cdots\omega_{d-m}(t_s)}=
\frac{\widetilde{\pi}}{\omega(t_1)\cdots\omega(t_s)}.$$ We note that $\deg_\theta(\widetilde{\pi}_d)=\frac{q}{q-1}$
and $\deg_\theta(\omega_d(t))=\frac{1}{q-1}$. 

We recall that we have set $\delta_{s,d}:=\frac{q^d-q}{q-1}-s\frac{q^{d-m}-1}{q-1}$ and that $\delta_{s,d}= m-1+\mu q^{d-m}$, and we stress that we do not use the connection with $H_{s,d}$ here.

\begin{Lemma}\label{lemmaident}
We have, in $K[\underline{t}_s]$, that:
$$\frac{\widetilde{\pi}_d}{\omega_{d-m}(t_1)\cdots\omega_{d-m}(t_s)}=-(-\theta)^{\delta_{s,d}-m+1}\frac{b_{d-m}(t_1)\cdots b_{d-m}(t_s)}{l_{d-1}},\quad d\geq m.$$
\end{Lemma}
\begin{proof}
We note that
\begin{eqnarray*}
\widetilde{\pi}_d&=&\theta(-\theta)^{\frac{1}{q-1}}\prod_{i=1}^{d-1}
\frac{-\theta^{q^i}}{\theta-\theta^{q^i}}\\ 
&=&-((-\theta)(-\theta)^{\frac{q^d-q}{q-1}}(-\theta)^{\frac{1}{q-1}})l_{d-1}^{-1}\\
&=&-(-\theta)^{\frac{q^d}{q-1}}l_{d-1}^{-1},\end{eqnarray*}
and that
\begin{eqnarray*}
\omega_d(t)&=&(-\theta)^{\frac{1}{q-1}}\prod_{i=0}^{d-1}\left(\frac{-\theta^{q^i}}{t-\theta^{q^i}}\right)\\
&=&(-\theta)^{\frac{1}{q-1}}(-\theta)^{\frac{q^d-1}{q-1}}b_d(t)^{-1}\\
&=&(-\theta)^{\frac{q^d}{q-1}}b_d(t)^{-1},
\end{eqnarray*}
so that 
\begin{eqnarray*}
\frac{\widetilde{\pi}_d}{\omega_{d-m}(t_1)\cdots \omega_{d-m}(t_s)}&=&-(-\theta)^{\frac{q^d}{q-1}}l_{d-1}^{-1}
(-\theta)^{-\frac{sq^{d-m}}{q-1}}b_{d-m}(t_1)\cdots b_{d-m}(t_s)\\
&=&-(-\theta)^{\frac{q^d}{q-1}-\frac{sq^{d-m}}{q-1}}\frac{b_{d-m}(t_1)\cdots b_{d-m}(t_s)}{l_{d-1}}\\
&=&-(-\theta)^{\delta_{s,d}-m+1}\frac{b_{d-m}(t_1)\cdots b_{d-m}(t_s)}{l_{d-1}},
\end{eqnarray*}
because $\frac{q^d}{q-1}-\frac{sq^{d-1}}{q-1}=\delta_{s,d}+\frac{q}{q-1}-\frac{s}{q-1}=\delta_{s,d}+\frac{q-s}{q-1}=\delta_{s,d}-m+1$. 
\end{proof}

The following lemma concludes the proof of Theorem \ref{theomathbs} and supplies the degree in $Y$ of $\HH_s$.

\begin{Lemma}\label{lemmalimit}
We have, in $\TT_s$:
$$\lim_{d\rightarrow\infty}\theta^{-q^{d-m}\mu}\HH_s(\theta^{q^{d-m}})=-\lambda_{1,s}.$$
Hence, the degree in $Y$ of $\HH_s$ equals $\mu :=  \frac{q^m-1}{q-1}-m$. 
\end{Lemma}
\begin{proof}
By Theorem \ref{anglespellarin}, we have $\lambda_{n,s}=\lambda_{1,s}\in A[\undt{s}]$, which is a polynomial of degree $m-1 = \frac{s-q}{q-1}$ in $\theta$, and by

Lemma \ref{lemmaident} that
\begin{eqnarray*}
\lambda_{1,s}&=& \frac{\omega(t_1)\cdots\omega(t_s)}{\widetilde{\pi}}\lim_{d\rightarrow\infty}F_d(1;s)\\
&=&\lim_{d\rightarrow\infty}\frac{\omega_{d-m}(t_1)\cdots\omega_{d-m}(t_s)}{\widetilde{\pi}_d}\frac{b_{d-m}(t_1)\cdots b_{d-m}(t_s)}{l_{d-1}}\HH_s(\theta^{q^{d-m}})\\
&=&-\lim_{d\rightarrow\infty}(-\theta)^{m-1-\delta_{s,d}}\HH_s(\theta^{q^{d-m}})\\
&=& -\lim_{d\rightarrow\infty}\theta^{-q^{d-m}\mu}\HH_s(\theta^{q^{d-m}})
\end{eqnarray*}
(observe that $\mu\in\ZZ$ is divisible by 2).

Finally, since we have shown that $\HH_s$ is a polynomial in $Y$, the claim on the degrees is now clear.
\end{proof}

\subsection{Proof of Theorem \ref{theoremhighercoefficients}} 
For convenience of the reader, we recall that we have set:
$$\Gamma_d=\frac{\prod_{i\geq d}\left(1-\frac{\theta}{\theta^{q^i}}\right)}{\prod_{i\geq d-m}\prod_{j=1}^s\left(1-\frac{t_j}{\theta^{q^i}}\right)}\in\TT_s(K_\infty).$$
We need the following result where we suppose that $s=1+m(q-1)$ with $m>0$.

\begin{Lemma}\label{gammasr}
For any integer $r$ with $0\leq r\leq \mu$ there exists a polynomial $$\Gamma_{s,r}\in A[\underline{t}_s,\theta][Y],$$ monic of degree $\mu-r$ in $Y$,
such that, for all $d\geq m$,  
\begin{equation}\label{gammar}
\theta^{(\mu-r)q^{d-m}}\Gamma_d=\Gamma_{s,r}(\theta^{q^{d-m}})+w_d,\end{equation}
where $(w_d)_{d\geq m}$ is a sequence of elements of $\TT_s(K_\infty)$ which tends to zero as
$d$ tends to infinity.
\end{Lemma}
\begin{proof}
We first need to focus on some general, elementary facts about formal series.
Let $f=1+\sum_{i\geq 1}f_iY^i$ be a formal series of $A[\underline{t}_s][[Y]]^\times$,
with coefficients $f_i\in\FF_q[\underline{t}_s]\subset A[\underline{t}_s]$. Then, both $f$ and its inverse $f^{-1}$
can be evaluated at any $Y=y\in\TT_s$ with $\|y\|<1$, where we recall that 
$\|\cdot\|$ denotes the Gauss norm of the Tate algebra $\TT_s$. We further note
that $f,f^{-1}$ never vanish in the disk of the $y$'s such that $\|y\|<1$.
Also, if $g=1+\sum_{i\geq 1}g_iY^i$ is a formal series of $A[\underline{t}_s][[Y]]^\times$
with coefficients $g_i\in A\subset A[\underline{t}_s]$ such that $\deg_\theta(g_i)=o(i^{\epsilon})$ for all $\epsilon>0$
(typically, a logarithmic growth), then, the evaluation of $g$ at any $Y=y\in\TT_s$
such that $\|y\|<1$ is allowed and yields a convergent series. Hence, the evaluation 
of the series $h=gf^{-1}\in A[\underline{t}_s][[Y]]^\times$ at $Y=y$ as above
is possible.

Additionally, if now $(y_d)_{d\geq d_0}$ is a sequence of elements of $\TT_s$ with $\|y_d\|<1$
such that $\|y_d\|\rightarrow0$ as $d\rightarrow\infty$, then, in virtue of the 
ultrametric inequality, for any integer $k$,
there exists a polynomial $H_k\in A[\underline{t}_s][Y]$ of degree $\leq k$ in $Y$, such
that
\begin{equation}\label{hd}
y_d^{-r}h(y_d)=H_r(y_d^{-1})+w_d,\end{equation} 
where $(w_d)_{d\geq d_0}$ is a sequence of elements of $\TT_s$ such that 
$\lim_{d\rightarrow\infty}w_d=0$.

We set, for all $d\geq0$, $y_d=\frac{1}{\theta^{q^d}}$.
We observe that, for all $d\geq 1$, 
$$\prod_{i\geq d}\left(1-\frac{\theta}{\theta^{q^i}}\right)=\sideset{}{'}\sum_{n\geq 0}(-\theta)^{\ell_q(n)}y_d^n,$$
where the sum is restricted to the integers $n$ which have, in their expansion in base $q$, only $0,1$ as 
digits. The growth of the degrees in $\theta$ of the coefficients of the formal series in $y_d$ for any fixed $d$
is then logarithmic.

Also, we have, for any integer $m\geq 0$ and $d\geq m$, 
$$\prod_{i\geq d-m}\left(1-\frac{t_j}{\theta^{q^i}}\right)=\sideset{}{'}\sum_{n\geq 0}(-t_j)^{\ell_q(n)}y_{d-m}^n.$$
In the preamble of the proof, we can thus set:
$$f:=\prod_{j=1}^s\sideset{}{'}\sum_{n_j\geq 0}(-t_j)^{\ell_q(n_j)}Y^{n_j},\quad g=\sideset{}{'}\sum_{n\geq 0}(-\theta)^{\ell_q(n)}Y^{q^mn}\in A[\underline{t}_s][[Y]]^\times$$
and we have, setting $h:=gf^{-1}\in A[\underline{t}_s][[Y]]^\times$, that
$h(y_{d-m})=\Gamma_d.$ The identity (\ref{gammar}) of the lemma now follows by using (\ref{hd})
with $k=\mu-r$ and the polynomial $\Gamma_{s,r}$ is easily seen to be monic of the claimed degree.
\end{proof}
\begin{Remark}{\em 
For example, $\Gamma_{s,1}=Y+t_1+\cdots+t_s$.}\end{Remark}

We can now prove Theorem \ref{theoremhighercoefficients}. We suppose that $d\geq m$ and that $m\geq 1$. We recall from
Lemma \ref{lemmaident} that 
\begin{eqnarray*}
\frac{1}{\Pi_{s,d}}&=&-(-\theta)^{\delta_{s,d}-m+1}\frac{\omega_{d-m}(t_1)\cdots\omega_{d-m}(t_s)}{\widetilde{\pi}_d}\\
&=&-(-\theta)^{m-1+\mu q^{d-m}-m+1}\frac{\omega_{d-m}(t_1)\cdots\omega_{d-m}(t_s)}{\widetilde{\pi}_d}\\
&=&-\theta^{\mu q^{d-m}}\frac{\omega_{d-m}(t_1)\cdots\omega_{d-m}(t_s)}{\widetilde{\pi}_d},
\end{eqnarray*}
because $\mu$ is even. 
After explicit expansion of (\ref{explicitidentity}), we write
$$\Pi_{s,d}^{-1}F_d(1;s)-\sum_{i=r+1}^\mu\mathbb{D}_i\theta^{iq^{d-m}}=
\mathbb{D}_r\theta^{rq^{d-m}}+\sum_{j=0}^{r-1}\mathbb{D}_j\theta^{jq^{d-m}}.$$
Dividing both sides by $\theta^{rq^{d-m}}$ we deduce that 
$$\mathbb{D}_r=-\theta^{(\mu-r)q^{d-m}}\frac{\omega_{d-m}(t_1)\cdots\omega_{d-m}(t_s)}{\widetilde{\pi}_d}F_d(1;s)-\sum_{i=1}^{\mu-r}\mathbb{D}_{i+r}\theta^{iq^{d-m}}+u_d,$$
with $u_d$ a sequence of elements of $\mathbb{T}_s(K_\infty)$ tending to zero.
We can rewrite this as
\begin{eqnarray*}
\mathbb{D}_r&=&-\theta^{(\mu-r)q^{d-m}}\frac{\omega(t_1)\cdots\omega(t_s)}{\widetilde{\pi}}\Gamma_dF_d(1;s)-\sum_{i=1}^{\mu-r}\mathbb{D}_{i+r}\theta^{iq^{d-m}}+u_d,\\
&=&-\Gamma_{s,r}(\theta^{q^{d-m}})\frac{\omega(t_1)\cdots\omega(t_s)}{\widetilde{\pi}}F_d(1;s)
-\sum_{i=1}^{\mu-r}\mathbb{D}_{i+r}\theta^{iq^{d-m}}+v_d,
\end{eqnarray*}
for another sequence of elements $v_d$ of $\TT_s(K_\infty)$ tending to zero 
as $d\rightarrow\infty$, by using (\ref{gammar}) of Lemma \ref{gammasr}.
The theorem follows.\CVD

\subsection{Proof of Theorem \ref{proppowersums}}\label{proofproppowersums}

For $a\in A^+(i)$ with $i\geq 0$ and for $j$ an index between $1$ and $s'$, we can write
$a(t_j)=t_j^i+b(t_j)$ where $b\in \FF_q[t]$ is a polynomial with degree in $t$ strictly smaller than $i$
(depending on $a$).
Hence, we can write, with $G_{s,d}$ a polynomial of $K[\undt{s}]$ such that for all $j=1,\ldots,s'$, 
$\deg_{t_j}(G_{s,d})<d-1$:
\begin{eqnarray*}
F_{s,d}&=&\sum_{i=0}^{d-1}\sum_{a\in A^+(i)}\frac{a(t_1)\cdots a(t_s)}{a}\\
&=&\sum_{a\in A^+(d-1)}\frac{a(t_1)\cdots a(t_{s'})t_{s'+1}^{d-1}\cdots t_s^{d-1}}{a}+G_{s,d}\\
&=&t_{s'+1}^{d-1}\cdots t_s^{d-1}S_{d-1}(1;s')+G_{s,d}.
\end{eqnarray*}
In particular, $S_{d-1}(1,s')\in K[\undt{s'}]$ is the coefficient of $t_{s'+1}^{d-1}\cdots t_s^{d-1}$ of the polynomial 
$F_{s,d}\in K[\undt{s}]$. Now, by Theorem \ref{Theorem2} we see that $S_{d-1}(1,s')$ is the coefficient
of $t_{s'+1}^{d-1}\cdots t_s^{d-1}$ in the polynomial $\Pi_{s,d}\mathbb{H}_s|_{Y=\theta^{q^{d-m}}}$
(for all $d\geq m$, with $s=m(q-1)+1$). Since the coefficient of $t_{s'+1}^{d-m}\cdots t_s^{d-m}$ in 
$b_{d-m}(t_1)\cdots b_{d-m}(t_s)$ is equal to $b_{d-m}(t_1)\cdots b_{d-m}(t_{s'})$,
we deduce that $S_{d-1}(1,s')$ is equal to $\Pi_{s',s,d}=\Pi_{s,d}/\prod_{i = s'+1}^s b_{d-m}(t_i)$ times the coefficient of 
$t_{s'+1}^{m-1}\cdots t_s^{m-1}$ of $\mathbb{H}_s|_{Y=\theta^{q^{d-m}}}$ which is 
$\mathbb{H}_{s,s'}|_{Y=\theta^{q^{d-m}}}$. \hfill \CVD

\subsection{Examples}\label{examplessmalls}
We give some formulas without proof.
If $s=q$, then $m=1$, and it can be proved that $\mathbb{H}_q=-1$ and $\lambda_{1,s}=1$.
If $s=2q-1$, we have that $m=2$, $\mu=q-1$ and the following formula holds, when $q>2$:
\begin{equation}\label{rudyformula}
\mathbb{H}_{2q-1}=\prod_{i=1}^{2q-1}(t_i-Y)+(Y^q-\theta)e_{q-1}(t_1-Y,\ldots,t_s-Y)\in A[\underline{t}_s],\end{equation}
where the polynomials $e_i$ are the elementary symmetric polynomials that, in the variables 
$T_1,\ldots,T_s$, is
defined by:
$$\prod_{i=1}^s(X-T_i)=X^s+\sum_{j=1}^s(-1)^jX^{s-j}e_j(T_1,\ldots,T_s).$$
Developing (\ref{rudyformula}) we obtain, again in the case $q>2$:
$$\mathbb{H}_{2q-1}=\sum_{i=0}^{q-1}(-1)^i(e_{2q-1-i}-\theta e_{q-1-i})Y^i$$
with $e_j=e_j(\undt{s})$, from which it is easy to see that $\mathbb{H}_{2q-1}$ has the following partial degrees: $\deg_{t_i}(\mathbb{H}_{2q-1})=1$ for all $i$, $\deg_\theta(\mathbb{H}_{2q-1})=1$ and $\deg_Y(\mathbb{H}_{2q-1})=\mu=q-1$ in agreement with the 
Theorem \ref{Theorem2}. Furthermore,
the coefficient of $Y^\mu$ is equal to
$-\mathbb{B}_{2q-1}=-\lambda_{1,2q-1}=-\theta+e_q(t_1,\ldots,t_{2q-1})$ (see the examples in \cite{APTR}). 

Also, computing the coefficient of the appropriate monomial in $t_1^{d-1},\ldots,t_s^{d-1}$ as in \S \ref{proofproppowersums} from the formula (\ref{rudyformula}), it is easy to deduce the formulas
$$S_{d-1}(1;s)=\frac{b_{d-1}(t_1)\cdots b_{d-1}(t_s)}{l_{d-1}}$$ for $d\geq 1$ and $s=0,\ldots,q-1$.
Further, we compute easily, for $d\geq 2$:
$$S_{d-1}(1,q) = \frac{b_{d-2}(t_1)\cdots b_{d-2}(t_q)}{l_{d-1}}\left(\prod_{i = 1}^q(t_i - \theta^{q^{d-2}}) + \theta^{q^{d-1}} - \theta \right),$$ which agrees with \cite[Corollary 4.1.8]{PER0}.
The analysis of the case $q=2$ 
is similar but we will not describe it here.
For $s=3q-2$ we have $m=3$ and $\mu=q^2+q$. We refrain from displaying here the explicit formulas we obtain.

\section{Properties of finite zeta values}\label{iota}

We first give the exact definition of the embedding $\iota$ in the statement of Theorem \ref{Theorem1}.
The ring $\AAA_s$ of the introduction is easily seen to be an $\boldsymbol{F}_s$-algebra (use the diagonal embedding
of $\boldsymbol{F}_s$ in $\AAA_s$).
The map $x\mapsto x^q$ induces an $\boldsymbol{F}_s$-linear automorphism of 
$\boldsymbol{F}_s[\theta]/P\boldsymbol{F}_s[\theta]$, hence, component-wise,
an $\boldsymbol{F}_s$-linear automorphism of $\AAA_s$ that we denote again by $\tau$.
There is a natural injective ring homomorphism
$$K\otimes_{\FF_q}\boldsymbol{F}_s\xrightarrow{\iota}\AAA_s$$
uniquely defined by sending $r\in K$ to the sequence of its reductions mod $P$, well defined for almost all primes $P$. 
Now, $\iota$ extends to $K^{1/p^\infty}\otimes_{\FF_q}\boldsymbol{F}_s$ in an unique
way by setting, for $r\in K^{1/p^\infty}\otimes_{\FF_q}\boldsymbol{F}_s$, 
$$\iota(r)=\tau^{-w}(\iota(\tau^w(r))),\quad w\gg0.$$

\begin{Lemma}\label{casesnot0}
Assuming that $n,s>0$ and that $n\not\equiv s\pmod{q-1}$, we have that $Z_\AAA(n;s)\in\AAA_s^\times$.
\end{Lemma}

\begin{proof} Let $P$ be a prime of $A$ of degree $d$,
let $v_P:K[\underline{t}_s]\rightarrow\ZZ\cup\{\infty\}$ be the $P$-adic valuation
normalized by $v_P(P)=1$. If $f\in K[\underline{t}_s]\setminus\{0\}$, then 
$v_P(f(\theta^{q^{k_1}},\ldots,\theta^{q^{k_s}}))\geq v_P(f)$, for any choice of $k_1,\ldots,k_s\in \ZZ$
non-negative integers. Let $k$ be a positive integer such that $N=sq^k-n>0$, and note that
since $n\not\equiv s\pmod{q-1}$, we have that $N\not\equiv0\pmod{q-1}$. Let $\operatorname{ev}:K[\underline{t}_s]\rightarrow K$ be the ring homomorphism defined by the substitution $t_i\mapsto\theta^{q^k}$
for all $i=1,\ldots,s$. On one hand, we have that, if $d$ is big enough,
$$\operatorname{ev}(F_d(n;s))=\operatorname{BG}(N;0):=\sum_{i\geq 0}\sum_{a\in A^+(i)}a^N\in A\setminus\{0\},$$
by \cite[Remark 8.13.8 1]{GOS} (because $N>0$ and $N\not\equiv0\pmod{q-1}$).
On the other hand, $$v_P(\operatorname{ev}(F_d(n;s)))\geq v_P(F_d(n;s)).$$ Assuming by contradiction
that $Z_\AAA(n;s)\not\in\AAA_s^\times$, there would exist an infinite sequence of primes $P$
such that $v_P(F_d(n;s))>0$. But then, the non-zero polynomial $\operatorname{BG}(N;0)$
would have infinitely many distinct divisors in $A$, hence yielding a contradiction.
\end{proof}

\subsection{Proof of Theorem \ref{Theorem1}}\label{proofoftheorem1} 

Let $P$ be a prime of $A$. We observe that, for $d=\deg_\theta(P),$ $l_{d-1}=(\theta-\theta^q)\cdots(\theta-\theta^{q^{d-1}})\equiv \left.\frac{P(t)}{t-\theta}\right|_{t=\theta}\equiv P'(\theta)\pmod{P}$. Also, $b_d(t)=(t-\theta)\cdots(t-\theta^{q^{d-1}})\equiv P(t)\pmod{P}$. Therefore:
\begin{eqnarray*}
\Pi_{s,d}&\equiv&\frac{P(t_1)\cdots P(t_s)}{P'(\theta)\prod_{i=1}^s\prod_{j=1}^{m}(t_i-\theta^{q^{d-j}})}\\
&\equiv&\frac{P(t_1)\cdots P(t_s)}{P'(\theta)\prod_{i=1}^s\prod_{j=1}^{m}(t_i-\theta^{q^{-j}})}\pmod{P}.
\end{eqnarray*}
Also, we have that $\mathbb{H}_s|_{Y=\theta^{q^{d-m}}}\equiv\mathbb{H}_s|_{Y=\theta^{q^{-m}}}\pmod{P}$.

We set
$$\mu_{1,s}:=-\frac{\mathbb{H}_s|_{Y=\theta^{q^{-m}}}}{\prod_{i=1}^s\prod_{j=1}^{m}(t_i-\theta^{q^{-j}})}\in K^{1/p^\infty}\otimes_{\FF_q}\boldsymbol{F}_s,$$
and this gives $Z_\AAA(1;s)=\widehat{\Pi}_{1,s}\iota(\mu_{1,s})$.

Now we tackle the non-vanishing of $\iota(\mu_{1,s})$. Let $d\geq 0$ be an integer, define the polynomial $\Psi_d(X) :=\frac{X^{q^d}-X}{X^q-X}$,
of degree $q^d-q$. Then, modulo $\Psi_d(X)$, the powers $1,X,X^2,\ldots,X^{q^d-q-1}\in\FF_q[X]$
are linearly independent over $\FF_q$.
We set now, for  $m\geq 1$ and for $d$ big enough, $w :=\mu q^{d-m}+m-1$, and we observe that,
again for $d$ big enough, $0\leq w<q^d-q$. Hence, the images of $1,\theta,\ldots,\theta^w$
in the ring $\FF_q[\underline{t}_s,\theta]/(\Psi_d)$ are $\FF_q$-linearly independent (where $\Psi_d=\Psi_d(\theta)$)
and a polynomial $H\in\FF_q[\underline{t}_s,\theta]$ of degree $\leq w$ in $\theta$ is zero modulo $\Psi_d$
if and only if it is identically zero.
Now, one easily shows using the degree in $Y$ of $\HH_s$ that for all $d$ sufficiently large, $-H_{s,d}=-\HH_s(\theta^{q^{d-m}})$ is a monic polynomial in $\theta$ of degree $w$.
In particular, the image of $-H_{s,d}$ in $\FF_q[\underline{t}_s,\theta]/(\Psi_d)$
is non-zero.

We now end the proof of the non-vanishing of $\mu_{1,s}$, equivalent to the non-vanishing of $\iota(\mu_{1,s})$, via a proof by contradiction. Let us suppose that for all $d$ big enough, and for any $P$ prime of $A$ of degree $d$,
we have $H_{s,d} \equiv 0 \pmod{P}$. In particular, this occurs for all large enough prime numbers $d =\varpi$. Now, $$\Psi_{\varpi}=\prod_{P;\deg_\theta(P)=\varpi}P,$$ and the reduction of $H_{s,\varpi}$ modulo $\Psi_\varpi$ is zero, giving the contradiction. Observe also that $\iota(\mu_{1,s})\in\AAA_s^\times$.
\CVD

\subsection{Proof of Theorem \ref{proposition1}}

We recall \cite[Lemma 4]{ANG&PEL} that the sum
$$S_{j,s}=\sum_{a\in A^+(j)}a(t_1)\cdots a(t_s)\in \FF_q[\undt{s}]$$
vanishes if and only if $j(q-1)>s$. In particular, for $N\geq 0$ and for all 
$j$ big enough, 
$S_j(-N;s)=\sum_{a\in A^+(j)}a^N a(t_1)\cdots a(t_s)=0$,
and the sum
$$\operatorname{BG}(N;s)=\sum_{j\geq 0}\sum_{a\in A^+(j)}a^N a(t_1)\cdots a(t_s)$$
represents a polynomial of $A[\undt{s}]$. 

First we motivate the proof. Noticing that
$$S_j(n;s)=\sum_{a\in A^+(j)}a^{-n}a(t_1)\cdots a(t_s)\equiv \sum_{a\in A^+(j)}a^{q^{\deg P}-1-n}a(t_1)\cdots a(t_s)\pmod{P}$$
for all primes $P$ such that $q^{\deg P}>n$, our goal is to prove that 
\[F_{\deg P}(n;s) \equiv \operatorname{BG}(q^{\deg P}-1-n;s) \pmod{P},\] for all irreducibles $P$ of large enough degree, under the assumption that $\ell_q(n) > s$. This amounts to showing that, if $\ell_q(n) > s$, then, for all $j \geq d$, we have 
\[\sum_{a\in A^+(j)}a^{q^{d}-1-n}a(t_1)\cdots a(t_s) = 0,\] and this follows immediately from \cite[Theorem 3.1]{PER2}. We give another proof here via a specialization argument employing \cite[Lemma 4]{ANG&PEL}. 

Fix integers $n \geq 1$ and $s \geq 0$, and assume $\ell_q(n) > s$. Let $d'$ be any integer such that $q^{d'} > n$. Set
$s':=\ell_q(q^{d'}-1-n)+s$. 
Develop $q^{d'}-1-n=\sum_{i=0}^{d'-1}c_iq^i$ in base $q$, so that the coefficients 
$c_i$ are in $\{0,\ldots,q-1\}$.
Letting $K[\undt{s'}]\xrightarrow{\operatorname{ev}} K[\undt{s}]$ be the 
ring homomorphism  obtained by sending the vector of 
variables $\undt{s'}=(t_1,\ldots,t_{s'})$ to the vector of values
$$(\underbrace{\theta,\ldots,\theta}_{c_0\text{ times}},\underbrace{\theta^q,\ldots,\theta^q}_{c_1\text{ times}},\ldots,\underbrace{\theta^{q^{d'-1}},\ldots,\theta^{q^{d'-1}}}_{c_{d'-1}\text{ times}},t_1,\ldots,t_s),$$
we have 
$$\operatorname{ev}(S_{j,s'}) = \sum_{a\in A^+(j)}a^{q^{d'}-1-n}a(t_1)\cdots a(t_s) \equiv S_j(n;s) \pmod{P},$$
for all primes $P$ of degree $d'$ with $q^{d'} > n$.
In particular, by Simon's Lemma, if $j(q-1)>s'$, then $S_{j,s'} = 0$ and, thus, $S_j(n;s)\equiv0\pmod{P}$, for all primes $P$ such that $q^{\deg P}>n$. Now, we notice that 
\[s' := \ell_q(q^{d'}-1-n)+s=\ell_q(q^{d'}-1)+s-\ell_q(n) = d'(q-1) + s-\ell_q(n),\] where the first equality holds as there is no base $q$ carry over in the sum $(q^{d'}-1-n) +n$. Hence, if we suppose now that
$d'=\deg P$ and $\ell_q(n)>s$, we find that $s' < j(q-1)$ for all $j \geq d'$, and hence
for all primes $P$ such that $q^{\deg P} > n$, we have
$$F_{\deg_\theta(P)-1}(n;s)\equiv\operatorname{BG}(q^d-1-n;s)\pmod{P}.$$

The vanishing result follows by observing that $\operatorname{BG}(q^d-1-n;s)=0$ if $q^d>n$ and
$n\equiv s\pmod{q-1}$, by \cite[Theorem 4.2]{PER2}. \CVD

\section{Further properties of finite zeta values}

\subsection{Period properties}\label{periodproperties}

We study a few additional properties of the elements of $\AAA_1$
$$\widehat{\pi}=\left(-\frac{1}{P'}\right)_P,\quad \widehat{\omega}=\left(\frac{1}{P(t)}\right)_P$$
that we present here as some kind of finite analogues of the elements $\widetilde{\pi}$
and $\omega(t)$.

The transcendence over $K$ of $\widetilde{\pi}$ and the transcendence of $\omega$
over $K(t)$ can be proved in a variety of ways (see for example the techniques of Papanikolas' \cite{PAP}). From this, we immediately deduce 
that $\widetilde{\pi}$ and $\omega$ are algebraically independent over the field 
$K(t)$. Now, we say that two elements $x,y\in \AAA_1$ are {\em algebraically independent}
if the only polynomial $Q\in K^{1/p^\infty}\otimes_{\FF_q}\boldsymbol{F}_s[X,Y]$ such that 
$Q(x,y)=0$ (for the $K^{1/p^\infty}\otimes_{\FF_q}\boldsymbol{F}_s$-algebra structure of $\AAA_1$) is the zero polynomial.
Similarly, we have a notion of element $x\in\AAA_1$ which is {\em transcendental} over $K^{1/p^\infty}\otimes_{\FF_q}\boldsymbol{F}_s$
(\footnote{These are not well behaved properties. Following a remark of M. Kaneko, it is possible
to construct, in $\AAA_1$, a non-zero element $x$ which is a root of a non-zero polynomial in 
$K\otimes_{\FF_q}\boldsymbol{F}_s[X]$ (hence, it is not ``transcendental") but also such that $x$ is not a root of any irreducible polynomial of
$K\otimes_{\FF_q}\boldsymbol{F}_s[X]$.}). We presently do not know if $\widehat{\pi}$ and $\widehat{\omega}$ are algebraically independent. Nevertheless, we can prove that $\widehat{\pi}$
is irrational (in other words, in $\AAA_0$, $\widehat{\pi}\not\in\iota(K)$) and that $\widehat{\omega}\in\AAA_1$
is transcendental.
\begin{Theorem}\label{irrational}
$\widehat{\pi}$ is irrational.
\end{Theorem}
\begin{proof}
Fix $d \geq 1$. For $1\leq k\leq d$, choose $(\alpha_1,\ldots,\alpha_k)\in\FF_q^k$ and relatively prime polynomials $f,g\in A$.
Denote by $\sharp(d)$ the cardinality of the set of primes $P\in A$ of degree $d\geq k$ such that
$P\equiv g\pmod{f}$ and $$\deg_\theta(P-\theta^d-\alpha_1\theta^{d-1}-\cdots-\alpha_k\theta^{d-k})<d-k.$$ 
We now invoke the following result of Hayes \cite{HAY}, strengthening Artin's analogue of the Prime Number Theorem 
for the field $K$. 
\begin{Theorem}[Hayes]\label{hayesTheorem}
We have
$$\sharp(d)=\frac{q^{d-k}}{d\Phi(f)}+O\left(\frac{q^{\vartheta n}}{n}\right),$$
with $0\leq \vartheta<1$, and where $\Phi$ is the function field analogue of Euler's $\varphi$-function 
relative to the ring $A$.{\em(\footnote{For the definition and the basic properties of $\Phi$, see Rosen, \cite[Chapter 1]{ROS}.})}
\end{Theorem}

In particular,
for all $n\geq 1$ fixed, there exist infinitely many primes $P$ of the form 
\begin{equation}\label{typeofprime}
P=\theta^{p\delta}+\theta^{p(\delta-n)+1}+\cdots
\end{equation}
(recall that $p$ is the characteristic of $\FF_q$).
For commodity, we denote by $\mathcal{P}_n$ the (infinite) set of primes $P$ of the form (\ref{typeofprime}).

Let us suppose, by contradiction, that there exists a $\kappa \in K$ such that $\iota(\kappa) = -1/\widehat{\pi}$, and write $\kappa = a/b$, with $a,b \in A$ coprime. Say $\deg b = n$. Then, using euclidean division, for all primes $P\in\mathcal{P}_n$, we may write
$$a=bP'+Q_PP,$$ for some $Q_P\in A$, since, by construction: $\deg(bP') < \deg P$ and $a \equiv bP' \pmod{P}$.
But the degree of the right side of the last displayed equation tends to infinity, while the degree of the left side is constant, giving the desired contraction. 
\end{proof}

\begin{Conjecture}
$\widehat{\pi}$ is transcendental over $K$.
\end{Conjecture}
This does not seem to follow from Hayes' Theorem \ref{hayesTheorem}. 

\begin{Remark}{\em 
We point out another 
perhaps interesting ``negative feature'' related to $\widehat{\pi}$: {\it it is not possible to interpolate the map $\FF_q^{ac} \ni \zeta \mapsto P_\zeta'(\zeta)\in\FF_q^{ac}$ by an element of $\TT$}, where $P_\zeta\in\FF_q[X]$ is the minimal polynomial of $\zeta$ and $\FF_q^{ac}$ denotes the algebraic closure of $\FF_q$ in $\CC_\infty$. In other words, there is no function $f\in\TT$ such that, for all
$\zeta\in\FF_q^{ac}$, $f(\zeta)=P'_\zeta(\zeta)$.
We skip the proof as this is superfluous for our purposes here.}\end{Remark}

\begin{Lemma}\label{transcendental}
$\widehat{\omega}$ is transcendental over $K(t)$.
\end{Lemma}
\begin{proof} 
Let us suppose that the statement is false. Then, there exist polynomials $a_0,\ldots,a_N\in A[t]$
such that $a_0a_N\neq0$, and such that for all but finitely many primes $P$, we have 
$$a_0+a_1P(t)+\cdots+a_NP(t)^N=r_PP,$$ for $r_P\in A[t]$. When the degree $d$ of $P$ is big enough,
we have $r_P=0$. Indeed, otherwise, the degree in $\theta$ of the left-hand side would be bounded from above
with the degree of the right-hand side tending to infinity. But then, there exists $P$ a prime such that
$r_P=0$ and such that the left-hand side does not vanish.
\end{proof}

\begin{Remark}{\em 
In particular, we deduce from the previous two results and Theorem \ref{Theorem1} that, for all integers $s\equiv 1\pmod{q-1}$, the finite zeta value $Z_{\AAA}(1;s)$ is not an element of $\iota(K\otimes_{\FF_q}\boldsymbol{F}_s)$.}

{\em The elements $\widehat{\pi}$ and $\widehat{\omega}$ should not be considered as exact analogues of $\widetilde{\pi}$ and $\omega$. To have closer analogues, we should multiply them by appropriate Carlitz $\theta$-torsion points. For example we have, in the convergent product defining the function of Anderson and Thakur $\omega$, the factor $(-\theta)^{\frac{1}{q-1}}$ which is a point of $\theta$-torsion for the Carlitz module.}

\end{Remark}

\section{Final remarks} \label{finalsection}

We come back to the $\QQ$-algebra $\AAA_\QQ$ and to the finite multiple zeta values, which we recall are defined, for $k_1,\ldots,k_r\in\ZZ$, by
$$\zeta_\AAA(k_1,\ldots,k_r)=\left(\sum_{0<n_1<\cdots<n_k<p}\frac{1}{n_1^{k_1}\cdots n_r^{k_r}}\pmod{p}\right)_p\in\AAA_\QQ.$$
Again, no one of these elements is currently known to be non-zero.

These elements have recently been the object of extensive investigation by Kaneko, Kontsevich, Hoffman, Ohno, Zagier, Zhao, et al. A strong motivation is the following conjecture. Let $Z_\AAA$ be the $\QQ$-sub-algebra of $\AAA_\QQ$ generated by the finite multi-zeta values and $Z_\RR$ be $\QQ$-sub-algebra of $\RR[T]$ generated by the ``renormalized multiple zeta values," 
which are polynomials of $\RR[T]$ with $T$ an indeterminate which algebraically represents the divergent series $\sum_{n>0}n^{-1}$ (described in the paper of Ihara, Kaneko and Zagier \cite{IKZ}). We consider the map $\eta:\{\zeta_\AAA(k_1,\ldots,k_r);r\geq 1,k_1,\ldots,k_r>0\}\rightarrow Z_\RR$
defined by 
\begin{equation}\label{eta}
\zeta_\AAA(k_1,\ldots,k_r)\mapsto\sum_{i=0}^r(-1)^{k_{i+1}+\cdots+k_r}\zeta(k_1,\ldots,k_i)\zeta(k_r,\ldots,k_{i+1}).\end{equation}
\begin{Conjecture}[Kaneko-Zagier, \cite{KAN}]\label{conjkanza}
The map $\eta$ induces a $\QQ$-algebra isomorphism
$$Z_\AAA\rightarrow\frac{Z_\RR}{\zeta(2)Z_\RR},$$
with $\zeta(2)=\sum_{n\geq 1}n^{-2}$. 
\end{Conjecture}

The following identities have been proved by Zhao \cite{ZHA}:
\begin{equation}
\label{zetak1k2}\zeta_\AAA(k_1,k_2)=\left((-1)^{k_1}\binom{k_1+k_2}{k_2}B_{p-k_1-k_2}\pmod{p}\right)_p.
\end{equation}
The identity (\ref{zetak1k2})
holds for all integers $k_1,k_2\geq 1$. We deduce the formula
$$\zeta_\AAA(1,k-1)=(B_{p-k}\pmod{p})_p\in\AAA_\QQ,\quad k>1\text{ odd},$$ where $B_n$ denotes the $n$-th Bernoulli number 
in $\QQ$ (\footnote{For $k=1$ the behavior is slightly different and Zhao's formula (\ref{zetak1k2}) does not hold. Instead, we have the formula $\zeta_\AAA(1,0)=1$ which is trivial. Note that, by the Theorem of Clausen-von Staudt, the denominator of $B_{p-1}$ is divisible by $p$ and we have
$p B_{p-1}\equiv-1\pmod{p}$ for all $p$.}).
We have the following conjecture, naturally related to Conjecture \ref{conjkanza}.

\begin{Conjecture}[Kaneko-Zagier]\label{conjkaza} 
For all odd integers $k\geq 3$, $(B_{p-k}\pmod{p})_p\in\AAA_\QQ$ is non-zero.
\end{Conjecture}

We now come back to our framework.
The {\it Bernoulli-Carlitz numbers} $BC_{j}$ may be defined in a variety of ways, for example as the coefficients multiplied by a Carlitz factorial of the reciprocal of the Carlitz exponential
\[\left( \sum_{i \geq 0} \frac{1}{D_i} z^{q^i - 1}  \right)^{-1} := \sum_{j \geq 0} \frac{BC_j}{\Pi_j} z^j,\]
as initially discovered by Carlitz \cite{CAR}; note that $BC_j \neq 0$ implies that $j \equiv 0 \pmod{q-1}$. One also has Carlitz's identity from ibid.
\[BC_j = \pitilde^{-j}{\Pi_j \zeta_A(j;0)},\]
which holds for all positive $j \equiv 0 \pmod{q-1}$. These two descriptions indicate a formal analogy with the classical Bernoulli numbers, but the analogy goes much deeper as demonstrated in Taelman's analog of the Herbrand-Ribet Theorem for the Carlitz module \cite{TAEhr} and several generalizations thereafter, e.g. in \cite{APTR}. 

Angl\`es, Ngo Dac and Tavares Ribeiro have informed us that they proved the following striking property, proving the analog to Conjecture \ref{conjkaza} in the Carlitz setting. If $s>1$ and $s\equiv1\pmod{q-1}$, the point $$\left(\operatorname{BC}_{|P|-s}\pmod{P}\right)_{P}\in\AAA_0,$$ not only is non-zero,
but is in fact a unit in $\AAA_0^\times$, hence confirming a conjecture in \cite{ANG&PEL};
their result is available in \cite{ANDTR}. There seems to be no analogue of (\ref{zetak1k2})
involving the Bernoulli-Carlitz fractions. Instead of this, certain congruences of similar flavor, but involving the Bernoulli-Goss polynomials, 
have been obtained in \cite{ANG&DOU} as well as by the first author. 
This suggests that we are 
far away from understanding what could be an analogue of Conjecture \ref{conjkanza}
in the Carlitzian framework, though this was the initial motivation of our investigations.

\end{document}